\documentclass[a4paper, 11pt]{article}

\usepackage[utf8]{inputenc}
\usepackage{amsmath}
\usepackage{amsthm}
\usepackage{amsfonts}
\usepackage{amssymb}
\usepackage{amstext}
\usepackage{dsfont}
\usepackage{color}
\usepackage[colorlinks]{hyperref}

\title{ Beta-gamma algebra identities and Lie-theoretic exponential functionals of Brownian motion}
\author{Reda \textsc{Chhaibi}        \footnote{Universit\"at Z\"urich. Email: \texttt{reda.chhaibi@math.uzh.ch}} } 
\date{}

\allowdisplaybreaks[4]

\DeclareMathOperator{\Inv}{Inv}

\def\half{\frac{1}{2}}
\def\eqlaw{\stackrel{\Lc}{=}}

\def\N{{\mathbb N}}

\def\Q{{\mathbb Q}}
\def\R{{\mathbb R}}
\def\C{{\mathbb C}}

\def\P{{\mathbb P}}
\def\E{{\mathbb E}}

\def\Ac{{\mathcal A}}

\def\Fc{{\mathcal F}}

\def\Lc{{\mathcal L}}
\def\Oc{{\mathcal O}}

\def\Uc{{\mathcal U}}

\def\afrak{{\mathfrak a}}

\def\gfrak{{\mathfrak g}}
\def\hfrak{{\mathfrak h}}

\def\nfrak{{\mathfrak n}}

\def\ufrak{{\mathfrak u}}

\newtheorem{thm}{Theorem}[section]
\newtheorem{proposition}[thm]{Proposition}
\newtheorem{corollary}[thm]{Corollary}
\newtheorem{question}[thm]{Question}

\newtheorem{definition}[thm]{Definition}
\newtheorem{example}[thm]{Example}
\newtheorem{lemma}[thm]{Lemma}

\newtheorem{rmk}[thm]{Remark}

\numberwithin{equation}{section}

\begin{document}
\date{}
\maketitle
\begin{center}
\textit{ In memoriam Marc Yor }
\end{center}

\begin{abstract}
We explicitly compute the exit law of a certain hypoelliptic Brownian motion on a solvable Lie group. The underlying random variable can be seen as a multidimensional exponential functional of  Brownian motion. As a consequence, we obtain hidden identities in law between gamma random variables as the probabilistic manifestation of braid relations. The classical beta-gamma algebra identity corresponds to the only braid move in a root system of type $A_2$. The other ones seem new.

A key ingredient is a conditional representation theorem. It relates our hypoelliptic Brownian motion conditioned on exiting at a fixed point to a certain deterministic transform of Brownian motion.

The identities in law between gamma variables tropicalize to identities between exponential random variables. These are continuous versions of identities between geometric random variables related to changes of parametrizations in Lusztig's canonical basis. Hence, we see that the exit law of our hypoelliptic Brownian motion is the geometric analogue of a simple natural measure on Lusztig's canonical basis.
\end{abstract}

\newpage
\tableofcontents

\section{Introduction}

Let $G$ be a complex semi-simple group of rank $r$. We fix a Borel subgroup $B$. $B = NH$ where $H \approx \left( \C^* \right)^r$ is a maximal complex torus and $N$ is a lower unipotent subgroup $N$. We denote by $\hfrak \approx \C^r$ the Lie algebra of $H$. If $\afrak \approx \R^r$ is the subspace where roots are real, we have $\hfrak = \afrak + i \afrak$. Moreover, we write $A := \exp\left( \afrak \right)$.

Since $\afrak$ is an Euclidean space thanks to the Killing form $\langle \cdot , \cdot \rangle$, there is a natural notion of Brownian motion on $\afrak$. Then the Brownian motion with drift $\mu$ is denoted by:
$$ X^{(\mu)}_t := X_t + \mu \ t$$

In the context of geometric crystals, the Robinson-Schensted correspondence with random input $X^{(\mu)}$ is performed by solving, on the Borel subgroup $B$, a left-invariant stochastic differential equation driven by this Brownian motion (\cite{bib:chh14a} section 10, \cite{bib:thesis}). One then obtains a Lie group valued stochastic process $B_t\left(X^{(\mu)}\right)$, which was first introduced by \cite{bib:BBO}. More is said in the preliminary section. We refer to this process as our hypoelliptic Brownian motion because its infinitesimal generator satisfies the (parabolic) H\"ormander condition. Although we will not make use of this fact, it is reassuring to know that it has a smooth transition kernel.

The stochastic process $B_t\left( X^{(\mu)} \right)$ has an $NA$ decomposition:
$$ B_t\left( X^{(\mu)} \right) = N_t\left( X^{(\mu)} \right) A_t\left( X^{(\mu)} \right)$$
The $A$ part is a multiplicative Brownian motion and does not converge. We focus on the $N$ part which plays the role of a multidimensional exponential functional of Brownian motion. Let $\Delta$ be set the set of simple roots and $C$ be the open Weyl chamber:
$$ C := \left\{ x \in \afrak \ | \ \forall \alpha \in \Delta, \langle \alpha, x \rangle > 0 \right\}$$
When $\mu \in C$ the $N$ part converges to $N_\infty\left( X^{(\mu)} \right)$. We refer to the law of $N_\infty\left( X^{(\mu)} \right)$ as the exit of law of the hypoelliptic Brownian motion. 

The first result is the conditional representation theorem \ref{thm:conditional_representation} that characterizes the law of a certain integral transform of Brownian motion as the Brownian motion $X^{(\mu)}$ conditioned to $N_\infty\left( X^{(\mu)} \right)$ being fixed. 

As a consequence, we are able to give an explicit formula for the exit law in theorem \ref{thm:N_infty_law}. This is our second main result. The expression involves independent gamma variables. In the case of the group $SL_2$, one recovers Dufresne's identity in law on the exponential functional of a Brownian motion with drift \cite{bib:Dufresne90}. For $W$ a standard Brownian motion and $\mu>0$, this identity states that the random variable $2\int_0^\infty e^{-2 W_s^{(\mu)}} ds$ has the same distribution as $\frac{1}{\gamma_\mu}$, where $\gamma_\mu$ is a gamma random variable with parameter $\mu$. For groups with higher rank, the presented construction gives the explicit law of multiple exponential functionals of Brownian motion.

Almost surely, $N_\infty\left( X^{(\mu)} \right)$ belongs to the set of totally positive matrices $N_{>0}^{w_0} \subset N$. The study of totally positivity in reductive groups has been initiated by George Lusztig (see \cite{bib:Lusztig08} for a survey), motivated by the theory of canonical bases. We will only need the fact that $N^{w_0}_{>0}$ possesses equivalent charts indexed by reduced words ${\bf i}$ of the longest element $w_0$ in the Weyl group. Simply by noticing that the law of $N_\infty\left( X^{(\mu)} \right)$ uses gamma variables and charts that depend on a choice of reduced word, we find hidden identities in law between these gamma variables. Primitive identities are associated to braid moves. It follows that we have as many primitive identities as there are rank $2$ root systems. That is the content of our third main result, stated as theorem \ref{thm:beta_gamma_identities}.

In fact, the identities between gamma variables tropicalize to identities between exponential random variables. We also prove the discrete version involving geometric random variables, using Lusztig's parametrization of canonical bases. 

\subsection*{Structure of the paper}
We begin by stating the three main theorems \ref{thm:conditional_representation}, \ref{thm:N_infty_law} and \ref{thm:beta_gamma_identities} in section \ref{section:main_results}, after the necessary preliminaries on Lie theory and total positivity. We will illustrate our claims thanks to examples from $SL_2$ and $SL_3$.

It is more convenient to postpone the proof of the conditional representation theorem \ref{thm:conditional_representation} and explain right away how it implies the two others. In section \ref{section:identities}, we show how rational identities between gamma variables tropicalize to min-plus identities between exponential variables. It implies the discrete version involving geometric random variables.

In section \ref{section:proof_conditional_thm}, we prove the conditional representation theorem. Thanks to results from \cite{bib:chh14a} and \cite{bib:thesis}, we reduce the problem to an induction whose base case is a result by Matsumoto and Yor on a relationship between Brownian motions with opposite drifts \cite{bib:MY01}. Just before diving into the proof, we explain how the Matsumoto-Yor theorem is the $SL_2$ case of ours.  We review this result for the sake of completeness in section \ref{section:review_matsumoto_yor}, along with Dufresne's identity. This section is absent in the published version.

Finally, we conclude with some open questions.

\subsection*{Acknowledgments}
The author is grateful to Marc Yor for fruitful discussions on exponential functionals of Brownian motion and beta-gamma identities. This paper is dedicated to his memory. I am also thankful to Philippe Bougerol for his guidance during my PhD thesis \cite{bib:thesis}, on which this article is based.

\section{Main results}
\label{section:main_results}

\subsection{Preliminaries}

\paragraph{Lie theory:} We will need some (mostly standard) notations and terminology for semi-simple groups and algebras (see for example \cite{bib:Springer09}, \cite{bib:Humphreys72}). Let $\gfrak$ be a complex semi-simple Lie algebra of rank $r$, and $\hfrak \approx \C^r$ is a maximal abelian subalgebra, the Cartan subalgebra. It has a Cartan decomposition $\gfrak = \nfrak \oplus \hfrak \oplus \ufrak$. $\Delta = \left\{ \alpha_i, \ 1 \leq i \leq r \right\} \subset \hfrak^*$ (resp. $\Delta^\vee \subset \hfrak$) denote the simple roots (resp. coroots) of $\gfrak$. The real part of the Cartan subalgebra $\afrak$ ( $\hfrak = \afrak + i \afrak)$ is the subspace where simple roots are real valued. The structure of $\gfrak$ is entirely encoded by the Cartan matrix $A = \left( a_{i,j} = \alpha_i\left( \alpha_j^\vee \right) \right)_{1 \leq i,j \leq r}$. The coefficients in the Cartan matrix determine the relations between Chevalley generators $\left( f_\alpha, h_\alpha = \alpha^\vee, e_\alpha \right)$, for $\alpha 
\in \Delta$. Simple roots form a basis of $\hfrak^*$ dual to the fundamental coweights $\left( \omega_\alpha^\vee \right)_{\alpha \in \Delta}$.

Let $G$ be a simply-connected complex Lie group with Lie algebra $\gfrak$. $N$, $H$ and $U$ are the subgroups with Lie algebras $\nfrak \oplus \hfrak \oplus \ufrak$. $H \approx \left( \C^* \right)^r$ is a maximal torus. $B = NH$ and $B^+=HU$ form a pair of opposite Borel subgroups. We have a one-parameter subgroup for every $\alpha \in \Delta$:
$$ \forall t \in \C, y_\alpha(t) := \exp(t f_\alpha),\ x_\alpha := \exp\left( t e_\alpha \right)$$

When convenient, they will sometimes be written $y_i$ and $x_i$, $ 1 \leq i \leq r$.

We identify $\hfrak$ and $\hfrak^*$ thanks to the Killing form $\langle \cdot , \cdot \rangle$. In this identification, $\alpha^\vee = \frac{2 \alpha}{\langle \alpha, \alpha \rangle}$. In general, we will write for any $\beta \in \hfrak^*$, $\beta^\vee = \frac{2 \beta}{\langle \beta, \beta \rangle} \in \hfrak$. The reflection on $\hfrak$ with respect to the hyperplane $\ker \beta$ is:
$$ \forall h \in \hfrak, s_\beta\left( h \right) = h - \beta\left( h \right) \beta^\vee$$

The Weyl group of $G$ is defined as $W = \textrm{Norm}( H )/ H $. It acts on the torus $H$ by conjugation and hence on $\hfrak$. As a Coxeter group, it is generated by the reflections $\left( s_\alpha \right)_{\alpha \in \Delta}$. Every $w \in W$ can be written as a product $s_{i_1} \dots s_{i_k}$ for a sequence ${\bf i} = \left( i_1, \dots, i_k \right)$. A reduced word for $w \in W$ is a sequence ${\bf i}$ of shortest possible length $\ell(w)$. The set of all possible reduced words for $w$ is denoted by $R(w)$. The Weyl group has a unique longest element denoted by $w_0$ and we set $m = \ell(w_0)$.

The Bruhat decomposition states that $G$ is the disjoint union of cells:
$$ G = \bigsqcup_{\omega \in W} B^+ \omega B^+  = \bigsqcup_{\tau \in W} B \tau B^+ $$

In the largest opposite Bruhat cell $B B^+ = N H U$, every element $g$ admits a unique Gauss decomposition in the form $g = n a u$ with $ n \in N$, $a \in H$, $u \in U$. In the sequel, we will write $g = [g]_- [g]_0 [g]_+$, $[g]_- \in N$, $[g]_0 \in H$ and $[g]_+ \in U$ for the Gauss decomposition. Also $[g]_{-0} := [g]_- [g]_0$.

\begin{rmk}
\label{rmk:sln_example}
The reader unfamiliar with Lie groups can have in mind the example of $SL_n(\C)$, of rank $r = n-1$. The following matrices can be chosen as Chevalley generators. If $E_{i,j} = \left( \delta_{i,r} \delta_{j,s} \right)_{1 \leq r, s \leq n}$ are the usual elementary matrices, then $h_i = E_{i,i} - E_{i+1,i+1}$, $e_i = E_{i,i+1}$, and $f_i = E_{i+1,i}$. $\hfrak$ is the set of complex diagonal matrices with zero trace, which we identify with $\left\{ x \in \C^n \ | \ \sum x_i = 0 \right\}$. Then $H$ is the set of diagonal matrices with determinant $1$. $N$ (resp. $U$) is the set of lower (resp. upper) triangular unipotent matrices. We have:
$$ \forall t \in \C, y_i(t) = \textrm{Id} + t E_{i+1,i}, \ x_i(t) = \textrm{Id} + t E_{i,i+1}$$

The Weyl group $W$ is the group of permutation matrices and acts on $\hfrak$ by permuting coordinates. In the identification with the symmetric group acting on $n$ elements, the reflections $s_i$ are identified with transpositions $\left( i  \ i+1 \right)$. The longest word $w_0$ reorders the elements $1, 2, \dots, n$ in decreasing order.

In the case of $GL_n$, the second Bruhat decomposition is known in linear algebra as the LPU decomposition which states that every invertible matrix can be decomposed into the product of a lower triangular matrix $L$, a permutation matrix $P$ and an upper triangular matrix $U$. $P$ is unique. The largest opposite Bruhat cell corresponds to $P = id$. It is dense as it is the locus where all principal minors are non-zero.
\end{rmk}

\begin{rmk}[Cartan-Killing classification]
Complex simple Lie algebras are known, and are classified by types:
\begin{itemize}
 \item Type $A_{r-1}$: $\gfrak = \mathfrak{sl}_r$ $G = SL_r(\C)$
 \item Type $B_r$: $\gfrak = \mathfrak{so}_{2r+1}$ $G = SO_{2r+1}(\C)$
 \item Type $C_r$: $\gfrak = \mathfrak{sp}_{r}$ $G = Sp_{r}(\C)$
 \item Type $D_r$: $\gfrak = \mathfrak{so}_{2r}$ $G = SO_{2r}(\C)$
 \item Exceptional types: $E_6$, $E_7$, $E_8$, $F_4$, $G_2$.
\end{itemize}
The number in subscript indicates the rank.
\end{rmk}

The set of roots is denoted by
$$\Phi := \left\{ \beta \in \hfrak^* \ | \ s_\beta \in W \right\} \ .$$
Roots split into positive and negative roots $\Phi = \Phi^+ \bigsqcup \Phi^-$, with $\Phi^+$ being the set of roots that can be written as a positive sum of simple roots. Reduced expressions of Weyl group elements give convex orderings of positive roots (see \cite{bib:Humphreys90}):
\begin{lemma}
\label{lemma:positive_roots_enumeration}
Let $(i_1, \dots, i_k)$ be a reduced expression of $w \in W$. Then for $j=1, \dots, k$:
$$\beta_{{\bf i}, j} := s_{i_1} \dots s_{i_{j-1}} \alpha_{i_j}$$
produces all the positive roots in the set of inversions of $w$:
$$ \Inv\left( w \right) := \left\{ \beta \in \Phi^+, \  w \beta \in \Phi^{-} \right\} \ .$$
For $w = w_0$, it produces all positive roots i.e $\Inv\left( w_0 \right) = \Phi^+$.
\end{lemma}
When ${\bf i} \in R(w_0)$, we call the sequence $\left( \beta_{{\bf i}, j} \right)_{1 \leq j \leq m} $ an enumeration of positive roots. When the chosen reduced expression is obvious from context, we will drop the subscript ${\bf i}$. In appendix, we list positive root enumerations for rank $2$ systems.

\paragraph{Total positivity:} In the classical case of $GL_n$, a totally non-negative matrix is a matrix such that all its minors are non-negative. Thanks to the classical Cauchy-Binet formula, totally non-negative matrices $\left( GL_n \right)_{\geq 0}$ form a semi-group. In \cite{bib:Lusztig94}, Lusztig generalized total positivity to reductive groups motivated by the theory of canonical bases. At that level of generality, taking the semi-group property as a definition is simpler. The totally non-negative part of $G$ is denoted $G_{\geq 0}$ and is defined as the semi-group generated by the following sets:

\begin{itemize}
 \item The semi-group $H_{>0} := A = exp\left( \afrak \right) \approx \left( \R_+^* \right)^r$
 \item The semi-group generated by $\left\{ x_\alpha(t), t>0, \alpha \in \Delta \right\}$: $U_{\geq 0}$
 \item The semi-group generated by $\left\{ y_\alpha(t), t>0, \alpha \in \Delta \right\}$: $N_{\geq 0}$
\end{itemize}

We will be interested in parametrizations of $G_{\geq 0}$. Lusztig proved that totally non-negative elements admit a Gauss decomposition made of totally non-negative elements.
\begin{thm}[\cite{bib:Lusztig94} Lemma 2.3]
\label{thm:totally_positive_gauss_decomposition}
Any element $g \in G_{\geq 0}$ has a unique Gauss decomposition $g = n a u$ with $n \in N_{\geq 0}$, $a \in A$ and $u \in U_{\geq 0}$.
\end{thm}

Therefore, one can focus on the parametrizations of $U_{\geq 0}$. The following theorem generalizes a result of Whitney \cite{bib:Whitney52} from $GL_n$ to $G$. It says that a totally positive element in $U$ can be written as a product of elementary matrices, depending on where it falls in the Bruhat decomposition.
\begin{thm}[\cite{bib:Lusztig94} Proposition 2.7, \cite{bib:BZ97} Proposition 1.1]
\label{thm:totally_positive_parametrizations}
For any $w \in W$ with $k=\ell(w)$, every reduced word $\mathbf{i} = (i_1, \dots, i_k)$ in $R(w)$ gives rise to a parametrization of $U^{w}_{>0} := U_{\geq 0} \cap B w B$ by:
$$
 \begin{array}{cccc}
x_{\mathbf{i}}: & \mathbb{R}_{>0}^k & \rightarrow & U^{w}_{>0}\\
                & (t_1, \dots, t_k) & \mapsto     & x_{i_1}(t_1) \dots x_{i_k}(t_k)
 \end{array}
$$
\end{thm}
We also define the transition maps:
$$R_{{\bf i}, {\bf i'}} := x_{\bf i'}^{-1} \circ x_{\bf i}$$
which play a particularly important role.

For the purposes of this paper, we will only need the case of $w=w_0$ and we will call such a parametrization the Lusztig parametrization of $U^{w_0}_{>0}$. The positive reals $\left( t_1, \dots, t_m \right)$ such that $u = x_{\bf i}\left( t_1, \dots, t_m \right) \in U^{w_0}_{>0}$ will be called Lusztig parameters of $u$, the dependence in ${\bf i}$ being implicit.

\paragraph{On Lusztig's canonical basis:} We will only be interested in the parametrizations of the canonical basis. The combinatorics of total posivity in the group $G$ are related to the combinatorics of Lusztig's canonical basis for the dual group, after a tropicalization procedure $\left[ . \right]_{trop}$. More is said on tropicalization in section \ref{section:identities} where $\left[ R_{\bf i, i'} \right]_{trop}$ is defined. For now, let us record the following:

\begin{thm}[see \cite{bib:Lusztig90a}, \cite{bib:Lusztig90b}]
\label{thm:lusztig_canonical_basis}
There is a canonical basis $\Lc$ of $\Uc_q\left( \nfrak^\vee \right)$, the lower unipotent part of the dual quantum group $\Uc_q\left( \gfrak^\vee \right)$. It has a natural parametrization associated to every reduced word ${\bf i} \in R\left( w_0 \right)$:
 $$ {\bf \Lc}_{\bf i}: \N^{\ell(w_0)} \longrightarrow \Lc$$
and the changes of parametrizations are given by:
 $$\left[ R_{\bf i, i'} \right]_{trop} = {\bf \Lc}_{\bf i'}^{-1} \circ {\bf \Lc}_{\bf i} \ .$$
\end{thm}

\paragraph{Brownian motion(s):} Since $\afrak$ is made into an Euclidean space thanks to the Killing form $\langle \cdot, \cdot \rangle$, there is a natural notion of Brownian motion on $\afrak$. Indeed, $\langle \cdot, \cdot \rangle$ is a scalar product once restricted to $\afrak$ and the induced norm is denoted by $||\cdot||$. Let $\left( \Omega, \Ac, \P \right)$ be the canonical probability space of a Brownian motion $X$ on $\afrak$. For $\mu \in \afrak$, the Brownian motion with drift $\mu$ is denoted by:
$$ X_t^{(\mu)} := X_t + \mu \ t \ .$$
Brownian motion with drift $\mu$ and starting position $x_0 \in \afrak$ is written:
$$ X_t^{(x_0, \mu)} := x_0 + X_t + \mu \ t \ .$$

One obtains a left-invariant Brownian motion $B_t(X^{(\mu)})$ on $B$ by solving the following stochastic differential equation driven by any path $X$. The symbol $\circ$ indicates that the SDE has to be understood in the Stratonovitch sense:
\begin{align}
\label{lbl:process_B_sde_stratonovich}
\left\{ \begin{array}{ll}
dB_t(X^{(\mu)}) = B_t(X^{(\mu)}) \circ \left( \sum_{\alpha \in \Delta} \half \langle \alpha, \alpha \rangle f_\alpha dt + dX^{(\mu)}_t\right) \\
 B_0(X^{(\mu)}) = \textrm{ Id }
\end{array} \right.
\end{align}

The SDE can be explicitly solved (\cite{bib:BBO}):
\begin{align}
\label{eqn:process_B_explicit}
B_t(X^{(\mu)}) & = 
\left( \sum_{ \substack{k \geq 0 \\ i_1, \dots, i_k } } f_{i_1} \dots f_{i_k} \int_{ t \geq t_k \geq \dots \geq t_1 \geq 0}  \prod_{j=1}^k dt_j \frac{|| \alpha_{i_j} ||^2}{2} e^{ -\alpha_{i_j}(X^{(\mu)}_{t_j}) } \right) e^{X^{(\mu)}_t}
\end{align} 

Clearly, $B_t(X^{(\mu)})$ has an $NA$ decomposition is given by:
\begin{align}
\label{eqn:process_N_explicit}
N_t(X^{(\mu)}) & = 
\sum_{ \substack{k \geq 0 \\ i_1, \dots, i_k } } \int_{ t \geq t_k \geq \dots \geq t_1 \geq 0} e^{ -\alpha_{i_1}(X^{(\mu)}_{t_1}) \dots -\alpha_{i_k}(X^{(\mu)}_{t_k}) } \frac{|| \alpha_{i_1} ||^2}{2} \dots \\
               & \quad \quad
\dots \frac{|| \alpha_{i_k} ||^2}{2} f_{i_1} \dots f_{i_k} dt_1 \dots dt_k
\end{align} 
\begin{align}
\label{eqn:process_A_explicit}
A_t\left( X^{(\mu)} \right) & = e^{X^{(\mu)}_t}
\end{align} 

\begin{rmk}
\label{rmk:ADE}
The presence of half squared norms $\frac{|| \alpha ||^2}{2}$ is here to account for the fact that time does not flow in the same fashion for all roots. In the simply-laced groups ($ADE$-types), one can choose all roots to be the same length, hence choosing $\frac{|| \alpha ||^2}{2} = 1$ for all $\alpha$.
\end{rmk}

\begin{rmk}
In the $A_1$ case only, we opt out of the normalization made in the previous remark. Indeed, the classical choice for the only root is $\alpha = 2$. Hence the factor $\frac{|| \alpha ||^2}{2} = 2$ in the following example of $SL_2$.
\end{rmk}

\begin{example}[$SL_2$ - $A_1$ type]
Let $X^{(\mu)}$ be a Brownian motion with drift $\mu$ on $\R$. The SDE is:
$$ dB_t(X^{(\mu)}) = B_t(X^{(\mu)}) \circ \begin{pmatrix} dX^{(\mu)}_t & 0\\ 2 dt & -dX^{(\mu)}_t \end{pmatrix}$$
and its solution is:
$$ B_t(X^{(\mu)}) = \begin{pmatrix} e^{X^{(\mu)}_t} & 0\\ e^{X^{(\mu)}_t} \int_0^t 2 e^{-2 X^{(\mu)}_s} \ ds & e^{-X^{(\mu)}_t} \end{pmatrix}$$
\end{example}

\begin{example}[$SL_n$ - $A_{n-1}$ type]
As in remark \ref{rmk:sln_example}, let $X$ be a Brownian motion with drift $\mu$ on $\left\{ x \in \R^n \ | \ \sum x_i = 0 \right\}$. For notational reasons, we drop the superscript $(\mu)$ and put indices as exponents. The SDE becomes:
$$ dB_t(X^{(\mu)}) = B_t(X^{(\mu)}) \circ \begin{pmatrix} dX^1_t & 0      & 0      & \cdots     & 0     \\
                                                          dt     & dX^2_t & 0      & \ddots     & \vdots\\
                                                          0      & dt     & \ddots & \ddots     & 0     \\
                                                          \vdots & \ddots & \ddots & dX^{n-1}_t & 0     \\
                                                          0      & \cdots & 0      & dt         & dX^n_t\\
                                           \end{pmatrix}$$
and its solution $B_t(X^{(\mu)})$ is given by:
$$ 
		    \begin{pmatrix} e^{X^1_t}                             & 0                                     & 0      & \cdots        \\
				    e^{X^1_t} \int_0^t e^{X^2_s-X^1_s} ds & e^{X^2_t}                             & 0      & \cdots        \\
				    e^{X^1_t} \int_0^t e^{X^2_{s_1}-X^1_{s_1}} ds_1 \int_0^{s_1} e^{X^2_{s_2}-X^1_{s_2}} ds_2 &
				    e^{X^2_t} \int_0^t e^{X^3_s-X^2_s} ds &
				    e^{X^3_t}                             & \ddots                 \\
				    \vdots                                & \vdots                                & \ddots & \ddots  
		      \end{pmatrix}$$
\end{example}

The process $A_t(X^{(\mu)})$ is a multiplicative Brownian motion on $A \approx \left( \R_+^* \right)^r$ and does not converge. The main theorems concern the $N$ part, which converges when the drift is inside the Weyl chamber, and we are able to give an explicit formula for the exit law. On a few occasions, we will need to consider the path $\pi$ driving the flow $B_.\left( \pi \right)$ to be any path. The previous construction carries verbatim if $\pi$ is a semi-martingale.

\subsection{Conditional representation theorem}

The symbol $\eqlaw$ stands for equality in law between random variables. Moreover, a generic gamma random variable with parameter $a>0$ is denoted $\gamma_a$.
$$ \P\left( \gamma_a \in dx \right) = \frac{1}{\Gamma(a)} x^{a-1} e^{-x} \mathds{1}_{\R_+}\left( x \right) dx$$
where $\Gamma$ is the gamma Euler function.

Let $C\left( \R_+, \afrak \right)$ be the space of continuous functions with values in $\afrak$. The following theorem defines our path transform:

\begin{thm}[ \cite{bib:BBO2} Proposition 6.4 , \cite{bib:chh14a} section 5 ]
\label{thm:path_transform_def}
For any totally non-negative $g \in U_{\geq 0}$, there is a path transform $T_g: C\left( \R_+, \afrak \right) \longrightarrow C\left( \R_+, \afrak \right)$ such that for any path $X$:
$$ \forall t \geq 0, B_t\left( T_g X \right) = \left[ g B_t\left( X \right) \right]_{-0}$$
This transform is well-defined in the sense that the Gauss decomposition always exists for $g \in U_{\geq 0}$.
\end{thm}

From now on, fix a reduced word ${\bf i} \in R(w_0)$ of length $m = \ell(w_0)$ and call $(\beta_1, \dots, \beta_m)$ the associated positive roots enumeration. Choose $\bar{w}_0$ to be any representative in $G$ on the longest element $w_0$ in $W = \textrm{Norm}( H )/ H$. Also, we introduce the shift vector:
$$\theta = \sum_{ \alpha \in \Delta } \log\left( \frac{\langle \alpha, \alpha\rangle}{2} \right) \omega_\alpha^\vee$$
where $\omega_\alpha^\vee$ are the fundamental coweights. As in remark \ref{rmk:ADE}, for $ADE$ groups, $\theta = 0$ because all roots can be chosen and are chosen to be of the same squared norm $2$. We can state the following theorem whose proof is postponed to section \ref{section:proof_conditional_thm}:

\begin{thm}[Conditional representation theorem]
\label{thm:conditional_representation}
Let $g \in U_{>0}^{w_0}$, $\mu \in C$ and $W$ a standard Brownian motion on $\afrak$.

Then $\Lambda^{x_0} := x_0 + T_{e^{\theta} g e^{-\theta}} \left( W^{(w_0 \mu)} \right)$ is distributed as a Brownian motion $X^{x_0, (\mu) }$
\begin{itemize}
 \item with drift $\mu$
 \item with initial position $x_0$
 \item conditioned on $N_\infty(X^{0, (\mu)}) = \Theta(g)$ where 
$$\Theta: U_{>0}^{w_0} \longrightarrow N_{>0}^{w_0}$$
 is the bijective function $\Theta(g) := [ g \bar{w}_0 ]_-$
\end{itemize}
Moreover, if we pick $g = x_{i_1}\left( t_1 \right) \dots x_{i_m}\left( t_m \right), m = \ell(w_0)$ being random with independent Lusztig parameters such that $t_j \stackrel{ \mathcal{L} }{=} \gamma_{ \langle\beta_{j}^\vee, \mu\rangle } $, then $\Lambda^{x_0}$ is a standard Brownian motion with drift $\mu$ starting at $x_0$.
\end{thm}

\begin{rmk}
The conditioning approach owes a lot to Baudoin and O'Connell \cite{bib:BOC09} in spirit, but ends up being quite different. Indeed, that paper considered a conditioning of Brownian motion $X^{(\mu)}$ with respect to simple integrals $\int_0^\infty e^{-\alpha(X^{(\mu)})}$, for $\alpha$ a simple root. Nevertheless, the geometric path model developed in \cite{bib:chh14a} and \cite{bib:thesis} makes it more natural to condition with respect to $N_\infty(X^{(\mu)})$. The random variable $N_\infty(X^{(\mu)})$ not only contains the simple integrals, but also iterated ones.  
\end{rmk}


\begin{example}[$\Theta$ in $A_1$-type]
\label{example:twist_A1}
For $G = SL_2$, one can choose:
 $$ \bar{w}_0 = \begin{pmatrix} 0 & -1\\ 1 & 0 \end{pmatrix} $$
Then:
 $$ \Theta\left( \begin{pmatrix} 1 & t\\ 0 & 1 \end{pmatrix} \right) = \left[ \begin{pmatrix} 1 & t\\ 0 & 1 \end{pmatrix} \begin{pmatrix} 0 & -1\\ 1 & 0 \end{pmatrix} \right]_- = \begin{pmatrix} 1 & 0\\ \frac{1}{t} & 0 \end{pmatrix}$$
\end{example}

\begin{example}[$\Theta$ in $A_2$-type]
\label{example:twist_A2}
For $G = SL_3$, one can choose:
 $$ \bar{w}_0 = \begin{pmatrix} 0 & 0       & 1      \\ 0 & -1 & 0  \\ 1 & 0 & 0 \end{pmatrix} $$
If:
$$ g = \begin{pmatrix} 1 & t_1 & 0\\ 0 & 1   & 0\\ 0 & 0 & 1 \end{pmatrix}
       \begin{pmatrix} 1 & 0   & 0\\ 0 & 1 & t_2\\ 0 & 0 & 1 \end{pmatrix}
       \begin{pmatrix} 1 & t_3 & 0\\ 0 & 1   & 0\\ 0 & 0 & 1 \end{pmatrix} $$
Then:
 $$ \Theta\left( g \right) =
\left[ \begin{pmatrix} 1 & t_1+t_3 & t_1 t_2\\ 0 &  1 & t_2\\ 0 & 0 & 1 \end{pmatrix} 
       \begin{pmatrix} 0 & 0       & 1      \\ 0 & -1 & 0  \\ 1 & 0 & 0 \end{pmatrix} \right]_- = 
\begin{pmatrix}
1                 & 0                             & 0 \\
\frac{1}{t_1}     & 1                             & 0\\
\frac{1}{t_1 t_2} & \frac{1+\frac{t_1}{t_3}}{t_2} & 1 
\end{pmatrix} $$

\end{example}

\subsection{Exit law}

Thanks to the conditional representation theorem, we obtain the exit law with little effort.

\begin{thm}[Law of $N_\infty(X^{(\mu)})$]
\label{thm:N_infty_law}
If $X^{(\mu)}$ is a Brownian motion with drift $\mu$ the Weyl chamber, then $N_t(X^{(\mu)})$ converges almost surely inside the open cell $N_{>0}^{w_0}$ and 
$N_\infty(X^{(\mu)}) = \Theta\left( x_{i_1}( t_1 ) \dots x_{i_m}( t_m ) \right)$ where the Lusztig parameters $t_j$ are independent random variables with:
$$ t_j \stackrel{\mathcal{L}}{=} \gamma_{ \langle\beta_{j}^{\vee}, \mu\rangle } $$
\end{thm}
\begin{proof}
The condition $\mu \in C$ entails the convergence of iterated integrals in the explicit expression of $N_t(X)$, equation \eqref{eqn:process_N_explicit}. The fact that $N_\infty(X^{(\mu)}) \in N^{w_0}_{>0}$ follows from a total posivity criterion (Theorem 1.5 in \cite{bib:BZ97} or Theorem 1.11 in \cite{bib:FZ99}) formulated in terms of generalized minors. Basically, one has only to adapt the proofs of theorem 5.2 in \cite{bib:chh14a}, or lemma 4.3 in \cite{bib:BBO} to the case of infinite time horizon.

The law of $N_\infty(X^{(\mu)})$ comes directly from theorem \ref{thm:conditional_representation}.

It is worth noting that the probability measure has a smooth density and charges the entire space $N^{w_0}_{>0}$, which is an open dense cell of $N_{\geq 0}$. Other cells are of smaller dimension and therefore of zero measure. Hence, it can be viewed as an $N_{\geq 0}$-valued random variable.
\end{proof}

Let us examine the example of $SL_2$, the group with smallest rank.
\begin{example}[$A_1$-type]
 For $G = SL_2$, $\afrak = \R$, $\alpha=2$, $\alpha^\vee=1$ and:
$$N_\infty(X^{(\mu)}) = \begin{pmatrix} 1 & 0\\ \int_0^\infty  2 e^{-2 X^{(\mu)}_s}ds & 1 \end{pmatrix} $$
We know that $\Theta\left( \begin{pmatrix} 1 & y \\ 0 & 1 \end{pmatrix} \right) = \begin{pmatrix} 1 & 0 \\ \frac{1}{y} & 1 \end{pmatrix}$, 
and as such theorem \ref{thm:N_infty_law} tells us that:
$$ \begin{pmatrix} 1 & 0\\ \int_0^\infty  2 e^{-2 X^{(\mu)}_s}ds & 1 \end{pmatrix} \eqlaw \begin{pmatrix} 1 & 0 \\ \frac{1}{\gamma_\mu} & 1 \end{pmatrix}$$
This is exactly Dufresne's identity in law \cite{bib:Dufresne90}.
\end{example}

Here, the law of $N_\infty(X^{(\mu)})$ can be seen as a meaningful generalization of Dufresne's identity which uses an inverse gamma. In the group setting, we see that the map $\Theta$ plays the role of the inverse map. Higher dimensional examples give new identities about exponential functionals of Brownian motion.

\begin{example}[$A_2$-type]
 For $G = SL_3$, $\afrak = \left\{ x \in \R^3 \ | \ x_1 + x_2 + x_3 = 0 \right\}$. Consider a Brownian motion $X^{(\mu)}$ on $\afrak$ with drift $\mu$ in the Weyl chamber. The simple roots are $\alpha_1 = \begin{pmatrix} 1 \\ -1 \\ 0 \end{pmatrix} , \alpha_2 = \begin{pmatrix} 0 \\ 1 \\ -1 \end{pmatrix} $. We obtain that:
$$ N_\infty(X^{(\mu)}) = \begin{pmatrix}
 1                                                                                  & 0                                           & 0 \\
 \int_0^\infty e^{-\alpha_1(X_s^{(\mu)})} ds                                        & 1                                           & 0\\
 \int_0^\infty e^{-\alpha_1(X_s^{(\mu)}) }ds \int_0^s e^{-\alpha_2(X_u^{(\mu)}) }du & \int_0^\infty e^{-\alpha_2(X_s^{(\mu)}) }ds & 1 
 \end{pmatrix} $$
We choose the reduced word ${\bf i} = (1,2,1) \in R(w_0)$. If:
$$\left(t_1, t_2, t_3 \right) \stackrel{\Lc}{=} \left( 
\gamma_{\langle \alpha^\vee_1                , \mu \rangle},
\gamma_{\langle \alpha^\vee_1 + \alpha^\vee_2, \mu \rangle},
\gamma_{\langle \alpha^\vee_2                , \mu \rangle}
\right)$$
are independent gamma random variables with corresponding parameters, then example \ref{example:twist_A2} and theorem \ref{thm:N_infty_law} tell us that:
\begin{align}
\label{eq:eqlaw_A2}
N_\infty(X^{(\mu)}) & \eqlaw 
\begin{pmatrix}
1                 & 0                             & 0 \\
\frac{1}{t_1}     & 1                             & 0\\
\frac{1}{t_1 t_2} & \frac{1+\frac{t_1}{t_3}}{t_2} & 1 
 \end{pmatrix} 
\end{align}
This result in itself seems new. It can be restated as follows in terms of the two positive parameters $a = \langle \alpha^\vee_2, \mu \rangle>$, $\widetilde{a} = \langle \alpha^\vee_2, \mu \rangle>0$ and the two correlated real Brownian motions:
$$ B_t := \frac{1}{\sqrt{2}}\left(X^1_t - X^2_t\right); \quad \widetilde{B}_t := \frac{1}{\sqrt{2}}\left(X^2_t - X^3_t\right)$$
The equality \eqref{eq:eqlaw_A2} is equivalent to the fact that the triple
$$ 
\left(
    \int_0^\infty e^{-\sqrt{2} B_s - a s} ds, \ 
    \int_0^\infty e^{-\sqrt{2} \widetilde{B}_s - \widetilde{a}s }ds, \ 
    \int_0^\infty e^{-\sqrt{2} B_s - a s } ds \int_0^s e^{-\sqrt{2} \widetilde{B}_u - \widetilde{a}u } du \right)
    $$
has the same distribution as $\left( \frac{1}{\gamma_{a}} , \ \frac{1+\frac{\gamma_a}{\gamma_{\widetilde{a}}}}{\gamma_{a+\widetilde{a}}}, \ \frac{1}{\gamma_{a} \gamma_{\widetilde{a}}} \right)$. 

Notice that the two first marginals are consistent with Dufresne's identity. In order to see that, one has to perform a Brownian rescaling of time by a factor $2$ and invoke the classical identity $ \frac{1+\frac{\gamma_a}{\gamma_{\widetilde{a}}}}{\gamma_{a+\widetilde{a}}} \eqlaw \frac{1}{\gamma_{\widetilde{a}}}$.
\end{example}

\subsection{Beta-gamma algebra}

We have then a natural notion of gamma law and inverse gamma law. Here $\gamma_.$ denote independent gamma random variables.

\begin{definition}[gamma law on $U$ and inverse gamma on $N$]
\label{def:group_gamma_law}
For $\mu \in C$ and $(\beta_1, \dots, \beta_m)$ being the positive roots enumeration associated to a reduced expression of $w_0 = s_{i_1} \dots s_{i_m}$, define $\Gamma_\mu$ to be the law of the positive (in the sense of total positivity) $U^{w_0}_{>0}$-valued random variable
$$ \Gamma_\mu \stackrel{\mathcal{L}}{=} x_{i_1}\left( \gamma_{ \langle\beta_{1}^{\vee}, \mu\rangle } \right) \dots x_{i_m}\left( \gamma_{ \langle\beta_{m}^{\vee}, \mu\rangle } \right) $$
Define the inverse gamma law on $N^{w_0}_{>0}$ as:
$$ D_\mu \stackrel{\mathcal{L}}{=} \Theta\left( \Gamma_\mu \right) \ .$$
\end{definition}

Those laws are well defined, in the sense that the above expressions do not depend on the choice of a reduced expression:
\begin{proposition}
\label{proposition:gamma_equality_in_law} 
If ${\bf i}$ and ${\bf i'}$ are reduced words of $w_0$, then for every $\mu \in C$, the following equality in law holds between independent gamma random variables:
$$ \left( \gamma_{ \langle\beta_{{\bf i'}, 1}^{\vee}, \mu\rangle } , \dots, \gamma_{ \langle\beta_{{\bf i'},m}^{\vee}, \mu\rangle } \right) \eqlaw R_{\bf i, i'}
   \left( \gamma_{ \langle\beta_{{\bf i }, 1}^{\vee}, \mu\rangle } , \dots, \gamma_{ \langle\beta_{{\bf i },m}^{\vee}, \mu\rangle } \right) $$
\end{proposition}
\begin{proof}
Theorem \ref{thm:N_infty_law} relates the above laws to the exit law of our hypoelliptic Brownian motion, which is intrinsically defined.
\end{proof}

The formula defining $\Gamma_\mu$ has more to it than meets the eye. Indeed, in order to have the law being the same for all reduced expressions of $w_0$, there must be hidden non-trivial equalities in law. These can be qualified as identities from the beta-gamma algebra, as some authors call them (see e.g \cite{bib:CPY98} and references therein). More is said in the next section. It is remarkable to think of them as a probabilistic manifestation of a group structure, more precisely of braid relationships.

If $s, s' \in W$ are reflections associated to simple roots and $d$ is the order of $ss'$, a braid relationship in $W$ is the equality between $d$ terms:
$$ s s' s \dots = s' s s' \dots $$
A braid move or a $d$-move occurs when substituting $ s s' s \dots$ for $s' s s' \dots $ within a reduced word. Considering two reduced expressions of the same Weyl group element, a well known result due to Hideya Matsumoto \cite{bib:Matsumoto64} and Tits \cite{bib:Tits69} states that one is obtained from the other by successive braid moves (see also \cite{bib:Humphreys90}).
$$ s_i s_j s_i \dots = s_j s_i s_j \dots $$
Using this fact, saying that $\Gamma_\mu$ is defined unambiguously is equivalent to saying that proposition \ref{proposition:gamma_equality_in_law} holds for reduced words which differ by a braid move. Hence for any $\mu \in \mathfrak{a}$, such that $\alpha_i(\mu)>0$ and $\alpha_j(\mu)>0$, one has:
\begin{align*}
       & x_i\left( \gamma_{\langle\alpha_i^\vee,\mu\rangle} \right) x_j\left( \gamma_{\langle s_i \alpha_j^\vee,\mu\rangle} \right) x_i\left( \gamma_{\langle s_i s_j \alpha_i^\vee,\mu\rangle} \right) \dots \\
\eqlaw & x_j\left( \gamma_{\langle\alpha_j^\vee,\mu\rangle} \right) x_i\left( \gamma_{\langle s_j \alpha_i^\vee,\mu\rangle} \right) x_j\left( \gamma_{\langle s_j s_i \alpha_j^\vee,\mu\rangle} \right) \dots 
\end{align*}
It turns out that the rank $2$ cases ($A_2$, $B_2$, $C_2$ and $G_2$) contain all the possible hidden identities. This is the classical rank $2$ reduction. Fix one of these root systems, and consider two reduced words ${\bf i}$, ${\bf i'}$  of the longest element $w_0$ in a rank $2$ system. The maps $R_{\bf i, i'}$ can be computed for every group explicitly and have been tabulated by Berenstein and Zelevinsky as theorem 5.3 in \cite{bib:BZ97}. Now write thanks to fundamental coweights $\mu = a_1 \omega_i^\vee + a_2 \omega_j^\vee$ and $(p_1, p_2, \dots) = R_{\bf i, i'}\left( t_1, t_2, \dots \right)$. By using the explicit formulas for the change of Lusztig parameters and the root enumerations given in tables \ref{tab:roots_A2}, \ref{tab:roots_B2}, \ref{tab:roots_C2} and \ref{tab:roots_G2}, we obtain the following theorem:

\begin{thm}[Beta-gamma algebra identities]
\label{thm:beta_gamma_identities}
Associated to every rank $2$ root system, the following identities in law hold:
\begin{itemize}

 \item Type $A_2$ (table \ref{tab:roots_A2}): $w_0 = s_1 s_2 s_1 = s_2 s_1 s_2$. ${\bf i} = (1,2,1)$ and ${\bf i'} = (2,1,2)$. Let $ p_1 = \frac{t_2 t_3}{t_1 + t_3}, p_ 2 = t_1 + t_3 , p_ 3 = \frac{t_1 t_2}{t_1 + t_3}$. Then:
$$(t_1, t_2, t_3) \stackrel{\mathcal{L}}{=}  \left( \gamma_{a_1}, \gamma_{a_1 + a_2}, \gamma_{a_2} \right)$$
\begin{center} if and only if \end{center}
$$(p_1, p_2, p_3) \stackrel{\mathcal{L}}{=}  \left( \gamma_{a_2}, \gamma_{a_1 + a_2}, \gamma_{a_1} \right) $$

 \item Type $B_2$ (table \ref{tab:roots_B2}): $w_0 = s_1 s_2 s_1 s_2 = s_2 s_1 s_2 s_1$. ${\bf i} = (1,2,1,2)$ and ${\bf i'} = (2,1,2,1)$. Let:
$$p_1 = \frac{t_2 t_3^2 t_4}{\pi_2}, p_2 = \frac{\pi_2}{\pi_1}, p_3 = \frac{\pi_1^2}{\pi_2}, p_4 = \frac{t_1 t_2 t_3}{\pi_1}$$
where $\pi_1 = t_1 t_2 + (t_1 + t_3)t_4, \pi_2 = t_1^2 t_2 + (t_1 + t_3)^2 t_4$. Then:
$$(t_1, t_2, t_3, t_4) \stackrel{\mathcal{L}}{=}  \left( \gamma_{a_1}, \gamma_{a_1 + a_2}, \gamma_{a_1 + 2 a_2}, \gamma_{a_2} \right)$$
\begin{center} if and only if \end{center}
$$(p_1, p_2, p_3, p_4) \stackrel{\mathcal{L}}{=}  \left( \gamma_{a_2}, \gamma_{a_1 + 2 a_2}, \gamma_{a_1 + a_2}, \gamma_{a_1} \right)$$

 \item Type $C_2$ (table \ref{tab:roots_C2}): $w_0 = s_1 s_2 s_1 s_2 = s_2 s_1 s_2 s_1$. ${\bf i} = (1,2,1,2)$ and ${\bf i'} = (2,1,2,1)$. Let:
$$p_1 = \frac{t_2 t_3 t_4}{\pi_1}, p_2 = \frac{\pi_1^2}{\pi_2}, p_3 = \frac{\pi_2}{\pi_1}, p_4 = \frac{t_1 t_2^3 t_3}{\pi_1}$$
where $\pi_1 = t_1 t_2 + (t_1 + t_3)t_4, \pi_2 = t_3 t_4^2 + (t_2 + t_4)^2 t_1$. Then:
$$(t_1, t_2, t_3, t_4) \stackrel{\mathcal{L}}{=}  \left( \gamma_{a_1}, \gamma_{2 a_1 + a_2}, \gamma_{a_1 + a_2}, \gamma_{a_2} \right)$$
\begin{center} if and only if \end{center}
$$(p_1, p_2, p_3, p_4) \stackrel{\mathcal{L}}{=}  \left( \gamma_{a_2}, \gamma_{a_1 + a_2}, \gamma_{2 a_1 + a_2}, \gamma_{a_1} \right)$$

 \item Type $G_2$ (table \ref{tab:roots_G2}): $w_0 = s_1 s_2 s_1 s_2 s_1 s_2 = s_2 s_1 s_2 s_1 s_2 s_1$. ${\bf i} = (1,2,1,2,1,2)$ and ${\bf i'} = (2,1,2,1,2,1)$. Let: 
$$ p_1 = \frac{t_2 t_3^3 t_4^2 t_5^3 t_6}{\pi_3}, p_2 = \frac{\pi_3}{\pi_2}, p_3 = \frac{\pi_2^3}{\pi_3 \pi_4} $$
$$ p_4 = \frac{\pi_3}{\pi_1 \pi_2}, p_5 = \frac{\pi^3_1}{\pi_4}, p_6 = \frac{t_1 t_2 t_3^2 t_4 t_5}{\pi_1}$$
where:
\begin{center}
 
\begin{align*}
 \pi_1 = & t_1 t_2 t_3^2 t_4 + t_1 t_2 (t_3 + t_5 )^2 t_6 + (t_1 + t_3 )t_4 t_5^2 t_6 
\end{align*}
\begin{align*}
 \pi_2 = & t_1^2 t_2^2 t_3^3 t_4 + t_1^2 t_2^2 (t_3 + t_5 )^3 t_6 + (t_1 + t_3 )^2 t_4^2 t_5^3 t_6\\
         & + t_1 t_2 t_4 t_5^2 t_6 (3t_1 t_3 + 2t_2 + 2t_3 t_5 + 2t_1 t_5 )
\end{align*}
\begin{align*}
\pi_3 = & t_1^3 t_2^2 t_3^3 t_4 + t_1^3 t_2^2 (t_3 + t_5 )^3 t_6 + (t1 + t3 )^3 t_4^2 t_5^3 t_6 \\
        & + t_1^2 t_2 t_4 t_5^2 t_6 (3t_1 t_3 + 3t_2 + 3t_3 t_5 + 2t_1 t_5 )
\end{align*}
\begin{align*}
\pi_4 = & t_1^2 t_2^2 t_3^3 t_4 ( t_1 t_2 t_3^3 t_4 + 2t_1 t_2 (t_3 + t_5 )^3 t_6 + (3t_1 t_3 + 3t_3^2 + 3t_3 t_5 + 2t_1 t_5 )t_4 t_5^2 t_6)\\
        & + t_6^2 ( t_1 t_2 (t_3 + t_5 )^2 + (t_1 + t_3 )t_4 t_5^2)^3 
\end{align*}
\end{center}

Then:
$$(t_1, t_2, t_3, t_4, t_5, t_6) \stackrel{\mathcal{L}}{=}  \left( \gamma_{a_1}, \gamma_{3 a_1 + a_2}, \gamma_{2 a_1 + a_2}, \gamma_{ 3 a_1 + 2 a_2}, \gamma_{a_1 + a_2}, \gamma_{ a_2} \right)$$
\begin{center} if and only if \end{center}
$$(p_1, p_2, p_3, p_4, p_5, p_6) \stackrel{\mathcal{L}}{=}  \left( \gamma_{ a_2}, \gamma_{a_1 + a_2}, \gamma_{ 3 a_1 + 2 a_2}, \gamma_{2 a_1 + a_2}, \gamma_{3 a_1 + a_2}, \gamma_{a_1} \right) $$

\end{itemize}
\end{thm}

\section{Identities in law}
\label{section:identities}

\subsection{Gamma identities}
Let us make some remarks on the obtained gamma identities. The classical beta-gamma algebra identity states that if $\gamma_a$ and $\gamma_b$ are two independent gamma random variables with parameters $a>0$ and $b>0$, then:
\begin{align}
\label{eqn:pair_gamma}
& \left( \frac{\gamma_a}{\gamma_a + \gamma_b}, \gamma_a + \gamma_b \right) 
\end{align}
form a pair of independent random variables.

This fact is easy to retrieve from the $A_2$ case. Indeed, by considering $(t_1, t_2, t_3) \eqlaw  \left( \gamma_{a_1}, \gamma_{a_1 + a_2}, \gamma_{a_2} \right)$ independent variables with the designated laws and $(p_1, p_2, p_3)$ algebraically defined as above, we know that $p_2$ and $p_3$ are independent. And since $p_2$ independent of $t_2$, we get that $p_2 = \gamma_{a_1} + \gamma_{a_2}$ is independent of $\frac{p_3}{t_2} = \frac{\gamma{a_1}}{\gamma_{a_1} + \gamma_{a_2}}$.

Moreover, Lukacs \cite{bib:Luk55} proved that the independence for the pair in \eqref{eqn:pair_gamma} characterizes the gamma law. This leads to the first open question in section \ref{section:open_questions}.

\subsection{Exponential identities}
A generic exponential random variable with parameter $\mu>0$ is denoted by ${\bf e}_{\mu}$:
$$ \P\left( {\bf e}_{\mu} \in dx \right) = \mu e^{-\mu x} \mathds{1}_{\R_+}\left(x\right) dx$$

Consider a rational expression in $k$ variables $a \in \Q\left( x_1, \dots, x_k \right)$ that has no minus sign. Tropicalizing $a$ to $\left[ a \right]_{trop}$ tantamounts to replacing the algebraic operations $\left( +, \times, / \right)$ by $\left( \min, +, - \right)$. A rational expression in the operations $\left( \min, +, - \right)$ is now commonly referred to as a tropical expression. Formally, if $a$ and $b$ are rational subtraction-free functions, then:
$$ \left[ a + b \right]_{trop} = \min( a, b) $$
$$ \left[ a   b \right]_{trop} = a + b $$
$$ \left[ a / b \right]_{trop} = a - b $$
$$ \left[ a \circ b \right]_{trop} = \left[ a \right]_{trop} \circ \left[ a \circ b \right]_{trop}$$

For example:
$$ \left[ \frac{t_2 t_3}{t_1 + t_3} \right]_{trop} = t_2 + t_3 - \min(t_1, t_3)$$

A less algebraic definition could be used, using a limit that always exists:
\begin{proposition}[Analytic tropicalization]
\label{proposition:analytic_tropicalization}
For a rational and subtraction-free expression in $k$ variables $a \in \Q\left( x_1, \dots, x_k \right)$, we have that for $h>0$:
\begin{align}
-h \log a\left( e^{-\frac{x_1}{h}}, \dots, e^{-\frac{x_k}{h}} \right) & = [a]_{trop}\left( x_1, \dots, x_k\right) + \Oc(h)
\end{align}
where $\Oc(h)$ is a quantity such that $\frac{\Oc(h)}{h}$ is bounded as $h \rightarrow 0$, uniformly in the variables $(x_1, \dots, x_k)$.
\end{proposition}
\begin{proof}
Let us prove the statement by induction on the size of the expression $a$, meaning the number of operations it uses (addition, multiplication and division). For the base case, notice that if $a$ is a monomial or a constant, then the statement is trivially true.
Now, for the inductive step, if $a$ is a product or ratio of two rational subtraction-free expressions, for which the statement is true, the statement follows using the properties of the logarithm. If $a$ is a sum whose terms satisfy the induction hypothesis:
$$ a = a_1 + a_2 \ ,$$
then for ${\bf x} = \left( x_1, \dots, x_k \right) \in \R^k$:
\begin{align*}
  & -h \log a\left( e^{-\frac{x_1}{h}}, \dots, e^{-\frac{x_k}{h}} \right) \\
= & -h \log\left( a_1\left( e^{-\frac{x_1}{h}}, \dots, e^{-\frac{x_k}{h}} \right) 
                + a_2\left( e^{-\frac{x_k}{h}}, \dots, e^{-\frac{x_k}{h}} \right) \right)\\
= & -h \log\left( e^{ \Oc(1)-h^{-1} [a_1]_{trop}( {\bf x} )}
                + e^{ \Oc(1)-h^{-1} [a_2]_{trop}( {\bf x} )} \right)\\
= & \min\left( [a_1]_{trop}( {\bf x} ), [a_2]_{trop}( {\bf x} )\right) + \Oc(h)\\
  & \ \ -h \log\left( 1 + e^{\Oc(1) - h^{-1} |[a_1]_{trop}-[a_2]_{trop}|( {\bf x} )} \right)\\
= & [a]_{trop}({\bf x}) + \Oc(h)
\end{align*}
\end{proof}

It is well known that $R_{\bf i, \bf i'}$ tropicalize to changes of parametrizations in Lusztig's canonical basis (Theorem \ref{thm:lusztig_canonical_basis}). Therefore, we consider the following crystallizing procedure for the rational subtraction-free expressions $R_{\bf i, i'}$. The use of the logarithm function has to be understood component-wise:
$$ \left[ R_{\bf i, i'} \right]_{trop} (u_1, u_2, \dots ) := \lim_{h \rightarrow 0 } - h \log R_{\bf i, i'}\left( e^\frac{-u_1}{h}, e^\frac{-u_2}{h}, \dots \right) $$

Now, on the probabilistic side, one can recover a tropical version of gamma identities involving exponential variables. The following lemma shows how gamma variables degenerate to exponential variables.
\begin{lemma}[Yor \cite{bib:Yor_private}]
\label{lemma:marc_yor}
 As $h>0$ goes to zero, $-h \log \gamma_{h \mu}$ converges in law to ${\bf e}_{\mu}$, an exponential variable 
with parameter $\mu$.
\end{lemma}
\begin{proof}
 A direct proof using densities is possible and straightforward. Let us rather present an aesthetically more pleasing derivation suggested by Marc Yor during a private communication. We can reduce the problem to $\mu=1$ because ${\bf e}_{\mu} \eqlaw \frac{1}{\mu} {\bf e}_{1}$ and:
 $$ \lim_{h \rightarrow 0} -h \log \gamma_{h \mu} = \frac{1}{\mu} \lim_{h \rightarrow 0} -h \log \gamma_{h}$$
 Now, if $\beta_{a,b}$ is a beta random variable with parameters $a>0$ and $b>0$, we can write $\gamma_{h}$ as a product of two independent random variables:
 $$ \gamma_h \eqlaw \beta_{h, 1} \gamma_{1+h}$$
 Knowing that $\left( \beta_{h, 1} \right)^h$ is in fact a uniform random variable, $-\log \beta_{h, 1}^h$ is distributed as ${\bf e}_{1}$. Hence:
 $$ -h \log \gamma_{h} \eqlaw {\bf e}_{1} - h \log \gamma_{1+h}$$
 The proof is finished upon noticing that $h \log \gamma_{1+h}$ converges to zero in probability.
 \end{proof}
 
Hence the tropical version of proposition \ref{proposition:gamma_equality_in_law}:
\begin{proposition}
\label{proposition:exponential_equality_in_law} 
If ${\bf i}$ and ${\bf i'}$ are reduced words of $w_0$, then for every $\mu \in C$, the following equality in law holds between independent exponential random variables:
$$ \left( {\bf e}_{ \langle\beta_{{\bf i'}, 1}^{\vee}, \mu\rangle } , \dots, {\bf e}_{ \langle\beta_{{\bf i'},m}^{\vee}, \mu\rangle } \right) \eqlaw \left[ R_{\bf i, i'} \right]_{trop}
   \left( {\bf e}_{ \langle\beta_{{\bf i }, 1}^{\vee}, \mu\rangle } , \dots, {\bf e}_{ \langle\beta_{{\bf i },m}^{\vee}, \mu\rangle } \right) $$
\end{proposition}
\begin{proof}
Before tropicalization, for the parameter $h \mu$, we have:
$$ \left( \gamma_{ \langle\beta_{{\bf i'}, 1}^{\vee}, h\mu\rangle } , \dots, \gamma_{ \langle\beta_{{\bf i'},m}^{\vee}, h\mu\rangle } \right) \eqlaw R_{\bf i, i'}
   \left( \gamma_{ \langle\beta_{{\bf i }, 1}^{\vee}, h\mu\rangle } , \dots, \gamma_{ \langle\beta_{{\bf i },m}^{\vee}, h\mu\rangle } \right) $$
Then, using the previous lemma:
\begin{align*}
& \left( {\bf e}_{ \langle\beta_{{\bf i'}, 1}^{\vee}, \mu\rangle } , {\bf e}_{ \langle\beta_{{\bf i'},2}^{\vee}, \mu\rangle }, \dots \right)\\
& \eqlaw \lim_{h \rightarrow 0} -h \log \left( \gamma_{ \langle\beta_{{\bf i'}, 1}^{\vee}, h\mu\rangle } , \gamma_{ \langle\beta_{{\bf i'},2}^{\vee}, h\mu\rangle }, \dots  \right) \\
& \eqlaw \lim_{h \rightarrow 0} -h \log R_{\bf i, i'} \left( \gamma_{ \langle\beta_{{\bf i}, 1}^{\vee}, h\mu\rangle } , \gamma_{ \langle\beta_{{\bf i},2}^{\vee}, h\mu\rangle }, \dots \right)\\
& =  \lim_{h \rightarrow 0} -h \log R_{\bf i, i'} \left( \exp\left( \frac{- h \log\left( \gamma_{h \langle\beta_{{\bf i}, 1}^{\vee}, \mu\rangle} \right) }{h} \right),
                                                         \exp\left( \frac{- h \log\left( \gamma_{h \langle\beta_{{\bf i}, 2}^{\vee}, \mu\rangle} \right)}{h}   \right), \dots \right)\\
& = \left[ R_{\bf i, i'} \right]_{trop}
   \left( {\bf e}_{ \langle\beta_{{\bf i }, 1}^{\vee}, \mu\rangle } , {\bf e}_{ \langle\beta_{{\bf i },2}^{\vee}, \mu\rangle }, \dots \right) 
\end{align*}
\end{proof}

Again, as all the information is contained in the rank two case, there is a finite list of identities between exponential variables that sums up the results so far. 
\begin{thm}[Exponential identities]
\label{thm:exponential_identities}
Associated to every rank $2$ root system, the following identities in law hold:
\begin{itemize}
 \item Type $A_2$ (table \ref{tab:roots_A2}): $w_0 = s_1 s_2 s_1 = s_2 s_1 s_2$. Let $ p_1 = t_2 + t_3 - \min(t_1, t_3), p_ 2 = \min( t_1, t_3) , p_ 3 = t_1 + t_2 - \min(t_1, t_3)$.Then:\\
$$(t_1, t_2, t_3) \stackrel{\mathcal{L}}{=}  \left( {\bf e}_{a_1}, {\bf e}_{a_1 + a_2}, {\bf e}_{a_2} \right)$$
\begin{center} if and only if \end{center}
$$(p_1, p_2, p_3) \stackrel{\mathcal{L}}{=}  \left( {\bf e}_{a_2}, {\bf e}_{a_1 + a_2}, {\bf e}_{a_1} \right) $$

 \item Type $B_2$ (table \ref{tab:roots_B2}): $w_0 = s_1 s_2 s_1 s_2 = s_2 s_1 s_2 s_1$. Let:
$$p_1 = t_2 + 2 t_3 + t_4 - \pi_2, p_2 = \pi_2 - \pi_1, p_3 = 2\pi_1 - \pi_2, p_4 = t_1 + t_2 + t_3 - \pi_1$$
where $\pi_1 = \min( t_1 + t_2, \min(t_1 , t_3) + t_4), \pi_2 = \min( 2 t_1 + t_2, 2\min(t_1 , t_3) + t_4 )$. Then:
$$(t_1, t_2, t_3, t_4) \stackrel{\mathcal{L}}{=}  \left( {\bf e}_{a_1}, {\bf e}_{a_1 + a_2}, {\bf e}_{a_1 + 2 a_2}, {\bf e}_{a_2} \right)$$
\begin{center} if and only if \end{center}
$$(p_1, p_2, p_3, p_4) \stackrel{\mathcal{L}}{=}  \left( {\bf e}_{a_2}, {\bf e}_{a_1 + 2 a_2}, {\bf e}_{a_1 + a_2}, {\bf e}_{a_1} \right) $$

 \item In the same fashion one can deduce the exponential identities for types $C_2$ and $G_2$.
\end{itemize}

\end{thm}

At this point, it seems very important to mention that in \cite{bib:BBO2}, a path model for Coxeter groups was developed, and exponential laws play a key role as infinima of a Brownian motion, with appropriate drift. There, once again, hidden identities in law can be found, involving general Coxeter braid relations. This goes beyond the crystallographic case we just considered.

This leads to the second open question in section \ref{section:open_questions}.

\subsection{Geometric identities}
Let a generic geometric random variable with parameter $0 < z < 1$ be denoted by ${\bf G}_z$:
$$ \forall k \in \N, \P\left( {\bf G}_z = k \right) = z^k \left( 1 - z \right)$$

Let ${\bf z} \in A$ such that $-\log{\bf z} \in C$. We write ${\bf z}^{\beta^\vee} = e^{\langle \beta^\vee, \log {\bf z} \rangle}$. The condition $-\log{\bf z} \in C$ entails that for every positive coroot $\beta^\vee \in \Phi^{+ \vee}$, $0<{\bf z}^{\beta^\vee}<1$. This allows us to formulate identities in law between geometric random variables.
\begin{proposition}
\label{proposition:geometric_equality_in_law} 
If ${\bf i}$ and ${\bf i'}$ are reduced words of $w_0$, then for every $\mu \in C$, the following equality in law holds between independent geometric random variables:
$$ \left( {\bf G}_{ {\bf z }^{ \beta_{{\bf i'}, 1}^{\vee} } } , \dots, {\bf G}_{ {\bf z }^{ \beta_{{\bf i'}, m}^{\vee} } } \right) \eqlaw \left[ R_{\bf i, i'} \right]_{trop}
   \left( {\bf G}_{ {\bf z }^{ \beta_{{\bf i }, 1}^{\vee} } } , \dots, {\bf G}_{ {\bf z }^{ \beta_{{\bf i }, m}^{\vee} } } \right) $$
\end{proposition}
\begin{proof}
The claim is true because $\left[ R_{\bf i, i'} \right]_{trop}$ is one-to-one from $\N^{m}$ to $\N^{m}$ and thanks to the elementary fact that:
$$ \lfloor {\bf e}_{\mu} \rfloor \eqlaw {\bf G}_z$$
with $z = e^{-\mu}$.
\end{proof}

It is immediate that the same identities as in theorem \ref{thm:exponential_identities} hold between geometric random variables. As a corollary, this allows us to prove that a natural analogue of $\Gamma_\mu$ (or $N_\infty\left( W^{(\mu)}\right)$) exists at the level of Lusztig's canonical basis. 
\begin{corollary}
For any ${\bf i} \in R\left( w_ 0 \right)$ with $\left( \beta_1, \dots, \beta_m \right)$ the associated positive root enumeration and $\left( {\bf G}_{ {\bf z }^{ \beta_1^{\vee} } } , \dots, {\bf G}_{ {\bf z }^{ \beta_m^{\vee} } } \right)$ independent geometric random variables, the random variable defined by
$$ C_z\left( \infty \right) := {\bf \Lc}_{\bf i}\left( {\bf G}_{ {\bf z }^{ \beta_1^{\vee} } } , \dots, {\bf G}_{ {\bf z }^{ \beta_m^{\vee} } } \right)$$
has a distribution on Lusztig's canonical basis that is independent of the choice of ${\bf i}$.
\end{corollary}

\section{Proof of the conditional representation theorem}
\label{section:proof_conditional_thm}

Let us state a version of the Matsumoto and Yor relationship between Brownian motions with opposite drifts, which itself is based on many previous works related to exponential functionals of Brownian motion, including Dufresne's identity.

\begin{thm}[ Matsumoto-Yor \cite{bib:MY01}, theorem 2.2 ]
\label{thm:my_bm_with_opposite_drifts_1}
Let $B^{(\mu)}$ be a Brownian motion on a Euclidean vector space $V \approx \R^n$ with drift $\mu$ and $\beta$ a linear form on $V$ such that $\beta(\mu)>0$. Denote by $s_\beta$ the hyperplane reflection with respect to $\ker \beta$, by $\Q^y$ the measure of Brownian motion conditionally on its exponential functional being equal to $y>0$:
$$\Q^y := \P\left( \ \cdot \ \Big | \int_0^\infty e^{-\beta( B_s^{(\mu)}) }ds = y \right)$$
and:
$$ \beta^\vee = \frac{2 \beta}{ \langle \beta, \beta \rangle }$$
Then:
$$ \hat{B}_t^{ (s_\beta \mu)} = B_t^{(\mu)} + \log\left(1-\frac{1}{y}\int_0^t e^{-\beta( B_s^{(\mu)}) }ds \right) \beta^{\vee} $$
is a $\Q^y$-Brownian motion with drift $s_\beta \mu = \mu - \beta(\mu) \beta^\vee$.
\end{thm}
This theorem has a dual version that characterises the reciprocal transform of Brownian motion as a Brownian motion conditioned with respect to its exponential functional:
\begin{thm}[ Matsumoto-Yor \cite{bib:MY01}, theorem 2.1]
\label{thm:my_bm_with_opposite_drifts_2}
Let $W^{(s_\beta \mu)}$ be a Brownian motion on $V \approx \R^n$ with drift $s_\beta \mu$, $\beta(\mu)>0$ and:
$$ X_t = W_t^{(s_\beta \mu)} + \log\left(1+\frac{1}{y}\int_0^t e^{ -\beta( W_s^{(s_\beta \mu)} ) }ds\right) \beta^{\vee} $$
Then $X$ is a Brownian motion with drift $\mu$, $B^{(\mu)}$, conditioned to:
$$\int_0^\infty e^{- \beta(B_s^{(\mu)})} ds = y .$$
If moreover, we pick $y$ as random with $ y \stackrel{\mathcal{L}}{=} \frac{2}{\langle \beta, \beta\rangle \gamma_{ \langle \beta^\vee, \mu \rangle } }$ independent from $W$, then $X$ is a Brownian motion with drift $\mu$.
\end{thm}

Notice that compared to the original formulation, we used a multidimensional setting. The change of sign is simply replaced by a hyperplane reflection. Let us now focus on proving the conditional representation theorem. In the case of $g=x_\alpha(\xi)$ with $\xi>0$, the group-theoretic path transform $T_g$ has a simple expression (\cite{bib:chh14a} Properties 5.19):
$$ \left(T_g X \right)_t = X_t + \log\left( 1 + \xi \int_0^t e^{-\alpha( X_s)} ds \right) \alpha^{\vee} $$
for any continuous path $X$. The result of Hiroyuki Matsumoto and Yor can be reformulated as:

\begin{thm}[$SL_2$ conditional representation]
\label{thm:sl2_conditional_representation}
If $g=x_\alpha(\xi)$, $\xi>0$, $W^{(s_{\alpha} \mu)}$ a Brownian motion on $\afrak$ with drift $s_{\alpha} \mu$ such that $\alpha(\mu)>0$ then $X := T_g(W^{(s_{\alpha} \mu)})$ is a Brownian motion with drift $\mu$ conditioned to $\int_0^\infty e^{ -\alpha( X_s ) }ds = \frac{1}{\xi}$.

If moreover we pick $\xi$ as random with $ \xi \stackrel{\mathcal{L}}{=} \frac{\langle \alpha, \alpha \rangle}{2} \gamma_{\langle \alpha^{\vee}, \mu \rangle }$ independent from $W$ then $X$ is a Brownian motion with drift $\mu$.
\end{thm}
\begin{rmk}
Recall that in the case of $SL_2$, $\Theta\left( \begin{pmatrix} 1 & \xi\\ 0 & 1 \end{pmatrix} \right) = \begin{pmatrix} 1 & 0\\ \frac{1}{\xi} & 1 \end{pmatrix}$. Hence, theorem \ref{thm:conditional_representation} becomes exactly theorem \ref{thm:sl2_conditional_representation} in the particular case of the group $SL_2$.
\end{rmk}

It is very surprising and impressive that Matsumoto and Yor fully worked out the $SL_2$ case without starting from any group-theoretic considerations. We are now ready to prove the conditional representation theorem.

\begin{proof}[Proof of theorem \ref{thm:conditional_representation} ]
We can of course take $x_0 = 0$. Let $X = T_{e^{\theta} g e^{-\theta}}( W^{(w_0 \mu)})$, and thanks to the composition property 5.19 in \cite{bib:chh14a} we have:
\begin{eqnarray*}
X & = & T_{e^{\theta} g e^{-\theta}} W^{(w_0 \mu)}\\
& = & T_{ x_{i_1}\left( \frac{\langle \alpha_{i_1}, \alpha_{i_1} \rangle}{2} t_1 \right) \dots 
          x_{i_m}\left( \frac{\langle \alpha_{i_m}, \alpha_{i_m} \rangle}{2} t_m \right) } W^{(w_0 \mu)}\\
& = & T_{ x_{i_1}\left( \frac{\langle \alpha_{i_1}, \alpha_{i_1} \rangle}{2} t_1 \right)} \circ \dots \circ 
      T_{ x_{i_m}\left( \frac{\langle \alpha_{i_m}, \alpha_{i_m} \rangle}{2} t_m \right) } W^{(w_0 \mu)}
\end{eqnarray*}
We apply inductively theorem \ref{thm:sl2_conditional_representation} with Lusztig parameters taken to follow the right laws, in order to get successive Brownian motions.

The end of proof follows from the deterministic inversion lemma for Lusztig parameters in \cite{bib:chh14a} theorem 8.16, which we know to be also valid for an infinite time horizon (\cite{bib:chh14a} subsection 8.5): 
$$ N_\infty\left( T_{e^{\theta} g e^{-\theta}}\left( W^{(w_0 \mu)} \right) \right) = \Theta\left( g \right)$$
Concerning notations, a little precision needs to be made at this point. The left-invariant flow $\left( B_t(.); t \geq 0 \right)$ considered in this paper is a conjugation by $e^{\theta}$ on the one in \cite{bib:chh14a} equation (5.7). Hence the little correction in the above formula. In the end:
$$ N_\infty\left( X^{0, (\mu)} \right) = \Theta\left( g \right)$$
concluding the proof.
\end{proof}

\section{Review of Dufresne's identity and the relationship proven by Matsumoto and Yor}
\label{section:review_matsumoto_yor}
This section is mainly expository in nature and reviews the probabilistic results we needed, namely Dufresne's identity and the proofs of theorems \ref{thm:my_bm_with_opposite_drifts_1} \ref{thm:my_bm_with_opposite_drifts_2}.

\subsection{Exponential functionals of Brownian motion}
An important fact is Dufresne's identity in law:
\begin{proposition}[ Dufresne \cite{bib:Dufresne90} ]
 \label{proposition:dufresne_identity}
 If $W^{(\mu)}$ is a one dimensional Brownian motion with drift $\mu>0$, then:
$$ \int_0^\infty e^{-2 W_s^{(\mu)} }ds \eqlaw \frac{1}{2\gamma_\mu} $$
\end{proposition}
\begin{proof}[Quick proof]
By time inversion, for any fixed $t>0$, the random variable $\int_0^t e^{-2 W_s^{(\mu)} }ds$ has the same law as $e^{-2 W_t^{(\mu)}} \int_0^t e^{ 2 W_s^{(\mu)} } ds$. Let $\left(Z_t; t>0\right)$ be given by:
$$ e^{-Z_t} := e^{-2 W_t^{(\mu)}} \int_0^t e^{ 2 W_s^{(\mu)} } ds$$
And, by Ito's lemma, $Z_t$ can be easily checked to be a diffusion process since it satisfies for $t>0$ the SDE:
$$ d Z_t = 2 dW^{(\mu)}_t - e^{Z_t}dt$$
Hence it has as infinitesimal generator:
$$ \Lc = 2 \partial_z^2 + (2 \mu - e^z) \partial_z $$
The sequence $Z_t$ converges in law to a unique invariant measure because $e^{-Z_t}$ has the same distribution as $\int_0^t e^{-2 W_s^{(\mu)} }ds$, which converges almost surely. This invariant measure will be the law of $\int_0^\infty e^{-2 W_s^{(\mu)} }ds$. Therefore, all we need to do is to prove that the distribution of $\log 2 \gamma_\mu$ is an invariant measure for $Z_t$. This is done easily by checking that the adjoint of $\Lc$ annihilates the density $p(z)$ of the random variable $\log 2\gamma_\mu$. We have:
$$ p(z) = \frac{1}{\Gamma(\mu) 2^\mu} \exp\left( \mu z - \half e^z \right)$$
Applying $\Lc^*$, the adjoint of $\Lc$:
$$ \Lc^* = 2 \partial_z^2 - \partial_z ( 2 \mu - e^z) $$
We get:
\begin{align*}
 \Lc^* p(z) & = 2 \partial_z^2 p(z) - 2 \partial_z \left( (\mu - \half e^z) p(z) \right)\\
& = 2 \partial_z^2 p(z) - 2 \partial_z^2 p(z)\\
& = 0
\end{align*}
\end{proof}

\subsection{Proofs of the relationship between Brownian motions of opposing drifts}
Let us state a version of the Matsumoto and Yor relationship between Brownian motions with opposite drifts $\cite{bib:MY01}$, which itself is based on many previous works related to exponential functionals of Brownian motion, including Dufresne's identity.

First let us start by proving theorem \ref{thm:my_bm_with_opposite_drifts_1} using known results on exponential functionals of Brownian motion. Now, let $B_t^{(\mu)}$ be an $n$-dimensional Brownian motion with drift $\mu$, $\Fc^B_t$ its natural filtration, $\beta$ a linear form such that $\beta(\mu)>0$ and:
$$N_t := \int_0^t \exp(-\beta( B_s^{(\mu)}) )ds $$
$$N_{\infty} = \lim_{t \rightarrow \infty} N_t $$

The law of $N_\infty$ comes as a simple corollary of Dufresne's identity:
\begin{corollary}
\label{corollary:N_infty_law_1d}
One has the identity in law:
$$ N_\infty :=  \int_0^\infty \exp\left(-\beta( B_s^{(\mu)}) \right)ds \stackrel{\mathcal{L}}{=} \frac{2}{||\beta||^2 \gamma_{\langle \beta^\vee, \mu \rangle } }$$
Therefore, the density is:
$$ \P\left( N_\infty \in dn \right) = \frac{1}{\Gamma(\langle \beta^\vee, \mu \rangle)} n^{-\langle \beta^\vee, \mu \rangle} \exp\left( -\frac{2}{||\beta||^2 n} \right) \left(\frac{2}{||\beta||^2 }\right)^{\langle \beta^\vee, \mu \rangle} \frac{dn}{n} \ .$$
\end{corollary}
\begin{proof}
Define the real Brownian motion $W$ by $\beta(B_t) = ||\beta|| W_t$ for $t \geq 0$. Then:
\begin{align*}
N_{\infty}
& = \int_0^\infty e^{- \beta( B_t^{(\mu)} ) }dt\\
& = \int_0^\infty e^{- ||\beta|| W_t - \beta(\mu)t }dt\\
& \stackrel{\mathcal{L}}{=} \int_0^\infty e^{- \frac{||\beta||}{\sqrt{c}} W_{tc} - \frac{\beta(\mu)tc}{c} }dt \textrm{ using Brownian scaling for } c>0\\
& = \frac{1}{c} \int_0^\infty e^{- \frac{||\beta||}{\sqrt{c}} W_{u} - \frac{\beta(\mu)u}{c} }du \textrm{ using change of variable } u = tc\\
& = \frac{4}{||\beta||^2} \int_0^\infty e^{- 2 W_{u} - 4\frac{\beta(\mu)u}{||\beta||^2} }du \textrm{ by choosing } c = \frac{ ||\beta||^2}{4}\\
& = \frac{4}{||\beta||^2} \int_0^\infty e^{- 2 (W_{u} + \langle \beta^\vee, \mu \rangle u ) }du \ .
\end{align*}
The result holds using Dufresne's identity in law. As for the density, for all $f\geq0$ bounded measurable function, and while writing $\nu = \langle \beta^\vee, \mu \rangle$, we have:
\begin{align*}
  & \E\left( f(N_\infty) \right)\\
= & \int_0^\infty f( \frac{1}{ ||\beta||^2 t/2}) \frac{ e^{-t} t^{\nu} }{\Gamma(\nu)} \frac{dt}{t}\\
= & \frac{1}{\Gamma(\nu)} \int_0^\infty f(n) \exp\left( -\frac{2}{||\beta||^2 n} \right) \left(\frac{2}{||\beta||^2 n} \right)^{\nu} \frac{2}{||\beta||^2} \frac{dn}{n} \textrm{ by letting } n = \frac{2}{||\beta||^2 t}\\
= & \frac{1}{\Gamma(\nu)} \int_0^\infty f(n) \exp\left( -\frac{2}{||\beta||^2 n} \right) \left(\frac{2}{||\beta||^2 } \right)^{\nu} n^{-\nu} \frac{dn}{n} \ .
\end{align*}
\end{proof}

\subsubsection*{Initial enlargement of the filtration $\Fc^B$ using the random variable $N_\infty$}
Now, we give a proof of theorem \ref{thm:my_bm_with_opposite_drifts_1} following Matsumoto and Yor's original presentation. In order to compute the law of $N_\infty$ conditionally on $\Fc^B_t$, the following decomposition is essential:
\begin{align}
\label{eqn:N_infty_decomposition_1d}
N_{\infty} = \int_0^t e^{- \beta( B_s^{(\mu)} ) }ds + e^{- \beta( B_t^{(\mu)} ) } \tilde{N}_\infty \ ,
\end{align}
with $\tilde{N}_\infty$ is a copy of $N_\infty$ independent from $\Fc^B_t$. Indeed:
\begin{align*}
N_{\infty}
& = \int_0^\infty e^{- \beta( B_s^{(\mu)} ) }ds\\
& = \int_0^t e^{- \beta( B_s^{(\mu)} ) }ds + e^{- \beta( B_t^{(\mu)} ) } \int_t^\infty e^{- \beta( B_s^{(\mu)} - B_t^{(\mu)}) }ds\\
& = \int_0^t e^{- \beta( B_s^{(\mu)} ) }ds + e^{- \beta( B_t^{(\mu)} ) } \tilde{N}_\infty \ .
\end{align*}

For readability purposes, and because it is not necessary to invoke general filtration enlargement theorems, we will give a complete proof using the usual tools. Indeed, as proved before the law of $N_\infty$ has a (smooth) density $\frac{d\P( N_\infty \leq y)}{dy}$ with respect to the Lebesgue measure, making possible the following computations.

Let $\Q^y = \P\left( \cdot | N_\infty = y \right)$ be a regular version of the conditional probability, $f \geq 0$ bounded measurable function, and $A \in \Fc^B_t$. We have:
\begin{align*}
  \int dy f(y) \Q^y(A) \frac{d\P}{dy}\left( N_\infty \leq y \right)
= & \E\left( \Q^{N_\infty}(A) f\left(N_\infty\right) \right)\\
= & \E\left( \mathds{1}_A f\left(N_\infty\right) \right)\\
= & \E\left( \mathds{1}_A \E\left( f\left(N_\infty\right) | \Fc^B_t\right) \right)\\
= & \E\left( \mathds{1}_A \int f(y) d\P(N_\infty \in dy | \Fc^B_t) \right)\\
= & \int dy f(y) \E\left( \mathds{1}_A \frac{ d\P(N_\infty \leq y | \Fc^B_t)}{dy} \right) \textrm{ (Fubini) } \ .
\end{align*}
Then:
$$ \Q^y(A) = \E\left( \mathds{1}_A \frac{ \frac{d\P}{dy}\left( N_\infty \leq y | \Fc^B_t \right) }{ \frac{d\P}{dy}\left( N_\infty \leq y \right) } \right) \ .$$
We conclude that $\Q^y$ is absolutely continuous with respect to $\P$ and that the Radon-Nikodym derivative on $\Fc^B_t$ is given by the $\P$-martingale:
\begin{align*}
q( B_t^{(\mu)}, N_t, y) := & \frac{d\Q^y}{d\P}_{|\Fc^B_t} \\
                         = & \frac{ \frac{d\P}{dy}\left( N_\infty \leq y | \Fc^B_t \right) }{ \frac{d\P}{dy}\left( N_\infty \leq y \right) }
\end{align*}
Using the expression for the density of $N_\infty$ from corollary \ref{corollary:N_infty_law_1d}, we get:
\begin{align*}
  & q( B_t^{(\mu)}, N_t, y)\\
= & \frac{ \frac{d\P}{dy}\left( \tilde{N}_\infty \leq (y-N_t) e^{\beta( B_t^{(\mu)})} | \Fc^B_t \right) }
         { \frac{d\P}{dy}\left( N_\infty \leq dy \right) } \\
= & e^{\beta( B_t^{(\mu)}) } \frac{ \exp\left( - \frac{2}{||\beta||^2(n-N_t)e^{\beta( B_t^{(\mu)}) } } \right) \left( (n-N_t)e^{\beta( B_t^{(\mu)})} \right)^{-(1+\langle \beta^\vee, \mu \rangle)} }
                                  { \exp\left( - \frac{2}{||\beta||^2 n} \right) n^{-(1+\langle \beta^\vee, \mu \rangle)} } \ .
\end{align*}
Hence:
$$ \log q( B_t^{(\mu)}, N_t, n) = A_t - \langle \beta^\vee, \mu \rangle \beta(B_t^{(\mu)}) - \frac{2 e^{-\beta( B_t^{(\mu)}) }}{||\beta||^2(n-N_t)}$$
where $A_t$ has a zero quadratic variation. Therefore, the semimartingale bracket between $\beta(B^{(\mu)})$ and $\log q$ is:
\begin{align*}
  \left\langle \beta(B^{(\mu)}), \log q \right\rangle_t
= & \left\langle \beta(B^{(\mu)}), - \langle \beta^\vee, \mu \rangle \beta(B_.^{(\mu)}) - \frac{2 e^{-\beta( B_.^{(\mu)}) }}{||\beta||^2(n-N_.)} \right\rangle_t\\
= & - \langle \beta^\vee, \mu \rangle||\beta||^2 - \frac{2}{||\beta||^2} \left\langle \beta(B), \frac{e^{-\beta( B_.^{(\mu)}) }}{n-N_.} \right\rangle_t\\
= & - 2\beta(\mu) + 2 \int_0^t \frac{e^{-\beta( B_s^{(\mu)}) }}{n-N_s} ds\\
= & - 2\beta(\mu) + 2 \int_0^t \frac{dN_s}{n-N_s}\\
= & - 2\beta(\mu) - 2 \log\left( 1 - \frac{N_t}{n} \right)
\end{align*}
In the end, using the Girsanov theorem (\cite{bib:RevuzYor} Chapter VIII, theorem 1.4):
$$ \hat{B}_t = B_t - \frac{\beta}{||\beta||^2} \langle \beta(B^{(\mu)}), \log q \rangle_t $$
is a $\Q^y$ Brownian motion, which completes the proof of theorem \ref{thm:my_bm_with_opposite_drifts_1}.

\subsubsection*{Inversion}
A natural question is whether we can recover $\hat{B}$ or $B$ from the other. The answer is yes and the argument is again due to Matsumoto and Yor \cite{bib:MY01}. The proof of theorem \ref{thm:my_bm_with_opposite_drifts_2} follows from \ref{thm:my_bm_with_opposite_drifts_1} and the following inversion lemma.
\begin{lemma}[Inversion lemma]
\label{lemma:inversion_lemma}
 Let $x$ and $y$ be $V$ valued paths i.e functions on $\R_+$. Then
$$ (1) \ x(t) = y(t) + \log( 1 + \frac{1}{n} \int_0^t e^{-\beta(y)}) \beta^{\vee}$$
if and only if
$$ (2) \left\{ \begin{array}{ll}
\forall t>0, \int_0^t e^{-\beta(x)} < n \\
y(t) = x(t) + \log( 1 - \frac{1}{n} \int_0^t e^{-\beta(x)}) \beta^{\vee}                
               \end{array} \right.
$$
Moreover, in any case:
$$ (3) \ ( 1 + \frac{1}{n} \int_0^t e^{-\beta(y)})( 1 - \frac{1}{n} \int_0^t e^{-\beta(x)}) = 1$$
and if $\int_0^\infty e^{ -\beta(y) } = \infty$ then $\int_0^\infty e^{-\beta(x)}=n$
\end{lemma}
\begin{proof}
 It is immediate to see that (1) and (2) are simultaneously true if and only if (3) is true. Then all we need to prove is $(1) \Rightarrow (3)$ and $(2) \Rightarrow (3)$
\begin{eqnarray*}
 (1) & \Rightarrow & e^{-\beta(x)} = \frac{ e^{-\beta(y)} }{ ( 1 + \frac{1}{n} \int_0^t e^{-\beta(y)})^2 }\\
& \Rightarrow & \frac{1}{n}\int_0^t e^{-\beta(x)} = \left[ \frac{-1}{ 1 + \frac{1}{n} \int_0^t e^{-\beta(y)} } \right]_0^t\\
& \Rightarrow & \frac{1}{ 1 + \frac{1}{n} \int_0^t e^{-\beta(y)} } = 1-\frac{1}{n}\int_0^t e^{-\beta(x)}\\
& \Rightarrow & (3)
\end{eqnarray*}
and
\begin{eqnarray*}
 (2) & \Rightarrow & e^{-\beta(y)} = \frac{ e^{-\beta(x)} }{ ( 1 - \frac{1}{n} \int_0^t e^{-\beta(x)})^2 }\\
& \Rightarrow & \frac{1}{n}\int_0^t e^{-\beta(y)} = \left[ \frac{1}{ 1 - \frac{1}{n} \int_0^t e^{-\beta(x)} } \right]_0^t\\
& \Rightarrow & \frac{1}{ 1 - \frac{1}{n} \int_0^t e^{-\beta(x)} } = 1+\frac{1}{n}\int_0^t e^{-\beta(y)}\\
& \Rightarrow & (3)
\end{eqnarray*}
Then (3) gives the convergence of $\int_0^t e^{-\beta(x)}$ to $n$, right away.
\end{proof}
Now we are ready to prove theorem \ref{thm:my_bm_with_opposite_drifts_2}:
\begin{proof}[Proof of theorem \ref{thm:my_bm_with_opposite_drifts_2}:]
Consider a Brownian motion $B^{(\mu)}$ with drift $\mu$ conditioned to $\int_0^\infty e^{- \beta(B_s^{(\mu)})} ds = n$. By the previous filtration enlargement argument, there is $\hat{B}$ a Brownian motion in the enlarged filtration, such that:
$$ \hat{B}_t^{(s_\beta \mu)} = B^{(\mu)}_t  + \log\left(1-\frac{1}{n}\int_0^t e^{-\beta( B_s^{(\mu)} ) }ds \right) \beta^{\vee} \ .$$
Using the inversion lemma:
$$ B^{(\mu)}_t = \hat{B}_t^{(s_\beta \mu)} + \log\left(1+\frac{1}{n}\int_0^t e^{ -\beta( \hat{B}_s^{(s_\beta \mu)} ) }ds\right) \beta^{\vee} \ .$$
Then the following equalities in law between processes follow:
\begin{align*}
  & \left( B^{(\mu)}_t; t \geq 0 | \int_0^\infty e^{- \beta(B_s^{(\mu)})} ds = n \right)\\
= & \left( \hat{B}_t^{(s_\beta \mu)} + \log\left(1+\frac{1}{n}\int_0^t e^{-\beta( \hat{B}_s^{(s_\beta \mu)} ) }ds \right) \beta^{\vee} ; t \geq 0 \right)\\
\stackrel{\mathcal{L}}{=} & \left( W_t^{(s_\beta \mu)} + \log\left(1+\frac{1}{n}\int_0^t e^{-\beta( W_s^{(s_\beta \mu)} ) }ds \right) \beta^{\vee} ; t \geq 0 \right)\\
= & \left( X_t; t \geq 0 \right) \ .
\end{align*}
This ends the proof of the first fact.

The second fact is just a consequence of knowing the law of $\int_0^\infty e^{- \beta(B_s^{(\mu)})} ds$ and the usual disintegration formula. For $F$ continuous functional on the sample space, we have:
$$ \E\left( F( B^{(\mu)}_. ) \right) =  \E\left( \E\left( F( B^{(\mu)}_.  ) | \int_0^\infty e^{ -\beta( B_s^{(\mu)} ) }ds \right) \right)$$
\end{proof}

\section{Some open questions}
\label{section:open_questions}

We end the paper by listing some open questions.

\begin{question}
We noticed in section \ref{section:identities} that the beta-gamma algebra identities in theorem \ref{thm:beta_gamma_identities} for type $A_2$ characterize the gamma random variable thanks to \cite{bib:Luk55}. Is it true in other types?

We would like to say that if the group theoretic transforms related to total positivity give independent random variables, then the input is made of gamma variables.
\end{question}

\begin{question}
In \cite{bib:BBO2}, for example in proposition 5.9, a sequence of mutually independent exponential random variables naturally appears. These depend on a choice of reduced expression, and one can deduce hidden identities in law identical to ours in the crystalligraphic case. However, the general Coxeter setting of \cite{bib:BBO2} goes beyond our framework, simply because for a non-crystalligraphic Coxeter group $W$ there is no Lie group for which $W$ is the Weyl group. 

Moreover, the tropical relations that appear there have irrational coefficients, and therefore cannot be the tropicalization of rational subtraction-free expressions. Indeed, in the proof of theorem 3.12, for the dihedral root system $I(m)$, transition maps make use of $\lambda = cos\left( \frac{2 \pi}{m} \right)$ and Tchebicheff polynomials in $\lambda$. $\lambda$ is indeed rational for the crystallographic values $m=2, 3, 4, 6$, but this is not true in general.

It would be very interesting to gain further insight in the Coxeter case and being able to explicit such relations. Is it possible to obtain a geometric lifting to identities between gamma variables?
\end{question}

\bibliographystyle{halpha}
\bibliography{Bib_Thesis}

\section*{Appendix: Positive root enumerations for rank \texorpdfstring{$2$}{2} systems}
\label{appendix:positive_roots_enumeration}

\begin{table}[ht!]
\centering
\label{tab:roots_A2}
$\alpha_1 = \begin{pmatrix} 1 \\ -1 \\ 0 \end{pmatrix} , \alpha_2 = \begin{pmatrix} 0 \\ 1 \\ -1 \end{pmatrix} $
  \begin{tabular}{|c|cc|}
  \hline
            &          121          &        212           \\
  \hline
  $\beta_1$ &      $\alpha_1$       &      $\alpha_2$      \\
  $\beta_2$ &   $\alpha_1+\alpha_2$ &  $\alpha_1+\alpha_2$ \\
  $\beta_3$ &      $\alpha_2$       &      $\alpha_1$      \\
  \hline
  \end{tabular}
\caption{Positive roots enumerations for type $A_2$}
\end{table} 

\begin{table}[ht!]
\centering
\label{tab:roots_B2}
$\alpha_1 = \begin{pmatrix} 1 \\ -1 \end{pmatrix} , \alpha_2 = \begin{pmatrix} 0 \\ 1 \end{pmatrix} $
  \begin{tabular}{|c|cc|}
  \hline
            &          1212         &        2121           \\
  \hline
  $\beta_1$ &      $\alpha_1$       &      $\alpha_2$       \\
  $\beta_2$ &   $\alpha_1+\alpha_2$ &  $\alpha_1+2\alpha_2$ \\
  $\beta_3$ &  $\alpha_1+2\alpha_2$ &   $\alpha_1+\alpha_2$ \\
  $\beta_4$ &      $\alpha_2$       &      $\alpha_1$       \\
  \hline
  \end{tabular}
\caption{Positive roots enumerations for type $B_2$}
\end{table} 

\begin{table}[!ht]
\centering
\label{tab:roots_C2}
$\alpha_1 = \begin{pmatrix} 1 \\ -1 \end{pmatrix} , \alpha_2 = \begin{pmatrix} 0 \\ 2 \end{pmatrix} $
  \begin{tabular}{|c|cc|}
  \hline
            &          1212         &        2121           \\
  \hline
  $\beta_1$ &      $\alpha_1$       &      $\alpha_2$       \\
  $\beta_2$ &  $2\alpha_1+\alpha_2$ &   $\alpha_1+\alpha_2$ \\
  $\beta_3$ &   $\alpha_1+\alpha_2$ &  $2\alpha_1+\alpha_2$ \\
  $\beta_4$ &      $\alpha_2$       &      $\alpha_1$       \\
  \hline
  \end{tabular}
\caption{Positive roots enumerations for type $C_2$}
\end{table} 

\begin{table}[!ht]
\centering
\label{tab:roots_G2}
$$\alpha_1 = \begin{pmatrix} 0 \\ 1 \\ -1\end{pmatrix} ,
 \alpha_2 = \begin{pmatrix} 1 \\ -2 \\ 1\end{pmatrix},
 $$
  \begin{tabular}{|c|cc|}
  \hline
            &             121212           &          212121              \\
  \hline
  $\beta_1$ &          $\alpha_1$          &          $\alpha_2$          \\
  $\beta_2$ &    $3\alpha_1+\alpha_2$      &       $\alpha_1+\alpha_2$    \\
  $\beta_3$ &     $2 \alpha_1+\alpha_2 $   &      $3\alpha_1 + 2\alpha_2$ \\
  $\beta_4$ &     $3 \alpha_1+2\alpha_2 $  &      $2\alpha_1+\alpha_2$    \\
  $\beta_5$ &     $\alpha_1+ \alpha_2$     &      $3\alpha_1+\alpha_2$    \\
  $\beta_6$ &          $\alpha_2$          &          $\alpha_1$          \\
  \hline
  \end{tabular}
\caption{Positive roots enumerations for type $G_2$}
\end{table} 

\end{document}